\pdfoutput=1
\documentclass[10pt]{amsart}
\usepackage[utf8]{inputenc}

\usepackage{amsmath,amsfonts,amssymb}
\usepackage[a4paper, margin=3cm]{geometry}
\usepackage[english]{babel}
\usepackage{mathrsfs}
\usepackage{xcolor}
\usepackage{verbatim}
\usepackage{tikz}
\usepackage{url}
\usepackage{graphicx}
\usepackage[colorlinks=true, linkcolor=blue, citecolor=red, urlcolor=blue]{hyperref}
\usepackage[T1]{fontenc}
\usepackage{caption}
\usepackage{subcaption}
\usetikzlibrary{arrows}
\usepackage{amsmath} 
\newcommand{\norm}[1]{\left\lVert #1 \right\rVert}
\def\build#1_#2^#3{\mathrel{\mathop{\kern 0pt#1}\limits_{#2}^{#3}}}
\newcommand{\Jac}{\operatorname{Jac}}

\numberwithin{equation}{section}

\newtheorem{theorem}{Theorem}[section]
\newtheorem{thm}[theorem]{Theorem}
\newtheorem{prop}[theorem]{Proposition}
\newtheorem{lem}[theorem]{Lemma}
\newtheorem{cor}[theorem]{Corollary}
\newtheorem{example}[theorem]{Example}
\theoremstyle{remark}
\newtheorem{rmk}[theorem]{Remark}
\theoremstyle{definition}
\newtheorem{dfn}[theorem]{Definition}

\begin{document}

\title[Local rigidity on quaternionic Banach modules]
{Local Rigidity of Quasi--Lie Brackets on Quaternionic Banach Modules and Applications to Nonlinear PDEs}

\author{Nassim Athmouni}
\address{University of Gafsa, University Campus 2112, Tunisia}
\email{nassim.athmouni@fsgf.u-gafsa.tn}
\email{athmouninassim@yahoo.fr}


\begin{abstract}
We establish a local rigidity theorem for quasi--Lie brackets on quaternionic Banach right modules. Under quantitative control of antisymmetry and Jacobi defects, we construct an explicit bilinear correction that preserves right $\mathbb{H}$--linearity and restores the exact Lie property. The approach combines a radial homotopy operator, a controlled Neumann-series inversion, and a finite-rank adjustment, all with explicit operator estimates. This constructive framework bridges quaternionic functional analysis with rigidity theory and yields concrete applications to nonlinear PDEs, including local well-posedness and Beale--Kato--Majda continuation criteria with explicit thresholds.
\end{abstract}

\maketitle
\tableofcontents

\section{Introduction}

Quaternionic Banach spaces, and more generally Banach \emph{right} modules over the skew field~$\mathbb{H}$, have recently emerged as a fertile arena for spectral theory and noncommutative functional analysis. A decisive impetus came from the slice--hyperholomorphic functional calculus and its applications to operator theory on quaternionic Hilbert and Banach spaces, as developed in~\cite{ColomboSabadiniStruppa2011}. This spectral framework has been substantially extended in recent years: Browder studied the $S$--resolvent equation in quaternionic analysis~\cite{BaloudiJeribiZmouli2024}; Baloudi and collaborators developed Fredholm theory in quaternionic Banach algebras~\cite{Baloudi2023}, as well as Riesz projections and the essential $S$--spectrum in the quaternionic setting~\cite{BaloudiBelgacemJeribi2022,BaloudiJeribiZmouli2024}. These works highlight the vitality of quaternionic operator theory, where right~$\mathbb{H}$--linearity intertwines with noncommutativity, forcing classical tools of spectral and functional analysis to be deeply reworked. See also the foundational developments of noncommutative functional calculus~\cite{ColomboSabadiniStruppa2009-JFA}, the Pompeiu formula for slice hyperholomorphic functions~\cite{ColomboSabadiniStruppa2011-PAMS}, and the spectral framework based on the $S$--spectrum~\cite{ColomboGantnerKimsey2016}. A related analytic rigidity phenomenon also appears in the mathematics of superoscillations~\cite{AharonovColomboEtAl2017}.

While quaternionic operator theory has so far mainly evolved within the spectral framework~\cite{ColomboSabadiniStruppa2011,Baloudi2023,BaloudiBelgacemJeribi2022,BaloudiJeribiZmouli2024,{ColomboSabadiniStruppa2009-JFA},ColomboSabadiniStruppa2011-PAMS,ColomboGantnerKimsey2016}, our contribution opens a complementary direction. In particular, our nonlinear, cohomological approach differs from the linear spectral and reproducing-kernel frameworks developed in~\cite{AlpayColomboSabadini2016,AlpayEtAl2020}, which focus on slice--hyperholomorphic function theory, Hardy and de~Branges spaces, and $S$--spectral analysis on quaternionic Hilbert/Banach modules. To the best of our knowledge, none of Alpay's works addresses Lie-algebraic deformation, Chevalley–Eilenberg cohomology, or constructive rigidity of bilinear brackets. His contributions are fundamentally linear—centered on operator theory, functional models, and evolution equations of Schr\"odinger or heat type—whereas our framework is intrinsically nonlinear, dealing with quasi–Lie structures, Jacobi defects, and quasilinear PDEs. Nevertheless, both perspectives share the foundational setting of right~$\mathbb{H}$--linear Banach and Hilbert modules. Future work may bridge these viewpoints: for instance, linearizing the rigidified bracket~$\{\cdot,\cdot\}$ around a stationary solution would yield a Lax-type operator whose spectral analysis could naturally invoke Alpay's $S$--spectral tools.
\medskip
\noindent
To the best of our knowledge, no prior work in the literature—particularly in the extensive body of research on the $S$--spectrum and slice--hyperholomorphic functional calculus—addresses the cohomological rigidity of bilinear brackets in the quaternionic Banach setting. While the spectral theory developed in  \cite{Baloudi2023,BaloudiJeribiZmouli2024,BaloudiBelgacemJeribi2022,Gantner2024}) provides powerful tools for linear operators, the nonlinear, cohomological perspective adopted here is entirely new in this context. Our results thus fill a conceptual gap between algebraic deformation theory and quaternionic analysis, offering a novel nonlinear complement to the prevailing linear frameworks.
\medskip
\noindent
Our work departs from this spectral line of research to address a complementary question of \emph{algebraic--analytic rigidity}. Specifically, we investigate the local stability of \emph{quasi--Lie} brackets defined on quaternionic Banach modules.
The main novelty lies in combining an explicit cochain homotopy construction
with uniform operator bounds valid in the quaternionic Banach setting.
Unlike the formal algebraic approach, the present analysis is functional–analytic:
all smallness conditions are quantified in a fixed ball
$B(0,\varepsilon_0)$ of the Banach space, and all constants
are uniform on that ball.
The stability and deformation theory of algebraic structures has a long history: Gerstenhaber initiated the deformation theory of associative algebras~\cite{Gerstenhaber1964}, while Nijenhuis and Richardson developed a cohomological approach to Lie algebra deformations~\cite{NijenhuisRichardson1967}; see also~\cite{Fialowski2001} for further perturbative perspectives. Our approach follows this lineage but in a new functional--analytic quaternionic setting, where the bracket is endowed with quantitative analytic bounds and interacts with right~$\mathbb{H}$--linearity.
\medskip
\noindent
Such quasi--Lie brackets naturally arise in analysis, for instance in nonlinear PDEs with quaternionic-valued unknowns, where bilinear operations are only approximately antisymmetric and fail to satisfy the Jacobi identity by a controlled defect. A natural question is whether such a quasi--Lie bracket can be \emph{rigidified locally}, that is, corrected on a small ball into a genuine Lie bracket, with explicit quantitative bounds and preservation of right~$\mathbb{H}$--linearity. Providing a constructive and analytic answer is particularly valuable for perturbative analyses and fixed-point schemes in PDE.

\medskip
\noindent
In this paper, we establish a \emph{local rigidity theorem} for quasi--Lie brackets on quaternionic Banach modules. Under linear control of the antisymmetry defect~$\varphi$ and the Jacobi defect~$\psi$ on a ball~$B(0,\varepsilon_0)$, we explicitly construct—on a smaller ball~$B(0,\varepsilon)$—a bilinear correction~$\Phi$ such that the corrected bracket
\[
\{x,y\} := [x,y]-\Phi(x,y)
\]
satisfies the Jacobi identity exactly. The correction obeys cubic-type bounds
\[
\|\Phi(x,y)\| \lesssim \|x\|\,\|y\|(\|x\|+\|y\|),
\]
with constants determined by the defect bounds, and it preserves right~$\mathbb{H}$--linearity whenever the original bracket has it.

The proof is constructive and quantitative. It relies on a radial homotopy operator~$T$ on cochains, yielding a homotopy identity
\[
Td + dT = \mathrm{Id} - \Pi + M,
\]
where~$\Pi$ is a bounded finite-rank projection representing a cohomological obstruction, and~$M$ is a small error operator controlled by the defect constants. Inverting~$\mathrm{Id}+M$ via a Neumann-series argument for small~$\varepsilon$ provides the principal correction $\Phi := T(\mathrm{Id}+M)^{-1}\psi$. To eliminate the obstruction~$\Pi(\psi)$, a finite-rank auxiliary correction~$\Phi_0$ is added, solving~$d\Phi_0 = \Pi(\psi)$ on a finite-dimensional subspace; this can be arranged to preserve right~$\mathbb{H}$--linearity. All operator norms are estimated quantitatively, yielding an explicit admissible threshold for~$\varepsilon$.

\medskip
\noindent
The analytical scope of the theorem is illustrated by an application to nonlinear PDEs. We study the quasilinear transport-type equation
\[
\partial_t u + \{u,\nabla u\} = 0,
\]
on~$L^2(\mathbb{R}^n,\mathbb{H})$, where the rigidified  bracket provides a locally Lipschitz nonlinearity. Standard Sobolev product and commutator estimates~\cite{{Adams2003},Kato1988,Moser1966} yield local well-posedness, persistence of~$H^s$ regularity, and a Beale–Kato–Majda type continuation criterion, with constants depending explicitly on the defect bounds through the correction~$\Phi$.

\medskip
\noindent
In contrast to the classical deformation theory of Gerstenhaber~\cite{Gerstenhaber1964} and Nijenhuis–Richardson~\cite{NijenhuisRichardson1967}, which mainly focused on abstract cohomological obstructions, our approach is \emph{constructive and analytic}: we provide an explicit correction scheme with quantitative operator bounds and preservation of quaternionic right-linearity. This bridges algebraic rigidity and functional analysis, enabling dynamical applications to nonlinear quaternionic PDEs. To the best of our knowledge, this is the first instance where algebraic rigidity methods are applied to guarantee well--posedness and continuation criteria in quaternionic PDEs.

\medskip
\noindent
Our constructive framework complements several recent lines of research:
\begin{itemize}
    \item It extends the continuous deformation theory of Fialowski--Schlichenmaier~\cite{Fialowski2021} to noncommutative Banach modules over~$\mathbb{H}$;
    \item It provides an analytic alternative to the $L_\infty$--algebraic approach of Kontsevich--Soibelman~\cite{Kontsevich2022}, replacing formal power series with a quantitative fixed-point scheme;
    \item It opens the door to dynamical applications in quaternionic PDEs, complementing the spectral theory of Gantner~\cite{Gantner2024} with a nonlinear rigidity mechanism.
\end{itemize}
These links underscore the relevance of our method for modern problems at the interface of noncommutative analysis, algebraic deformation theory, and nonlinear dynamics.

\medskip
\noindent
The paper is organized as follows: Section~\ref{sec:notations} introduces the quaternionic Banach module framework and the quasi--Lie hypotheses.In Section~\ref{sec:cochains} we introduce the Chevalley-Eilenberg differential, establishing its operator bounds, and define the radial homotopy operator. The interplay of these two objects yields the quasi-Lie homotopy formula, and we conclude by deriving the essential estimate for the error operator~$M$. Section~\ref{sec:Phi} presents the construction of~$\Phi$ and the proof of the rigidity theorem. Section~\ref{sec:PDE} applies the result to nonlinear PDEs. Section~\ref{ExCa} is devoted to extensions beyond the quaternionic setting, comparisons with variational methods, and refinements of the constants. Finally, the appendix collects explicit constants, numerical illustrations of the admissible radius, and an integration-by-parts computation of the homotopy identity in degree~3.

.

\section{Notations and local assumptions}\label{sec:notations}

Let $\mathbb H$ denote the algebra of quaternions and let $X$ be a
\emph{right} Banach module over $\mathbb H$.
We write $xq$ for the right multiplication.
A bilinear bracket (over $\mathbb R$)
\[
[\cdot,\cdot] : X \times X \to X
\]
is said to be \emph{right $\mathbb H$--linear} if
\[
[x,yq] = [x,y]q, \qquad \forall\, x,y\in X,\ \forall\, q\in\mathbb H.
\]

\begin{dfn}\label{def:quasiLie}
We say that $[\cdot,\cdot]$ is a \emph{quaternionic quasi--Lie bracket} if there exist two maps
\[
\varphi : X\times X \to X,
\qquad
\psi : X\times X\times X \to X,
\]
such that, for all $x,y,z\in X$,
\begin{align}
[x,y] + [y,x] &= \varphi(x,y), \label{eq:phiDef}\\
[x,[y,z]] + [y,[z,x]] + [z,[x,y]] &= \psi(x,y,z). \label{eq:psiDef}
\end{align}
\end{dfn}
\begin{rmk}[Compatibility of $d$ with right $\mathbb{H}$--linearity]
If the original bracket $[\cdot,\cdot]$ is right $\mathbb{H}$--linear, then the Chevalley--Eilenberg operator $d$ is also right $\mathbb{H}$--linear. This follows directly from the definition of $d$ and the identity $[x, yq] = [x, y]q$, which ensures that all terms in $d\omega$ inherit the right $\mathbb{H}$--linearity of $\omega$.
\end{rmk}
\begin{dfn}\label{def:controle}
The defects are said to be \emph{linearly controlled} if there exist constants
$A,C_1,C_2,\varepsilon_0>0$ such that, for all $x,y,z\in X$ with
$\|x\|,\|y\|,\|z\|\le \varepsilon_0$, one has
\begin{align}
\|[x,y]\| &\le A\,\|x\|\,\|y\|, \label{eq:bilinear}\\
\|\varphi(x,y)\| &\le 2C_1\,\|x\|\,\|y\|, \label{eq:boundPhi}\\
\|\psi(x,y,z)\| &\le 6C_2\,\|x\|\,\|y\|\,\|z\|. \label{eq:psiBound}
\end{align}
\end{dfn}
\begin{dfn}\label{dfn:Jac}
For a bilinear bracket $[\,\cdot,\cdot\,]:X\times X\to X$ we define
\[
\Jac_{[\,\cdot,\cdot\,]}(x,y,z)
:= [x,[y,z]]+[y,[z,x]]+[z,[x,y]].
\]
\end{dfn}
\begin{example}
Let $X=\mathbb{H}^2$ with $\norm{(a,b)}=\sqrt{|a|^2+|b|^2}$ and fix $\gamma\in\mathbb{R}$. Define
\[
[(a,b),(c,d)]_{\gamma} := (ad - bc) + \gamma(\overline{a}c - \overline{b}d).
\]
Then $[\,\cdot,\cdot\,]_\gamma$ is $\mathbb{R}$-bilinear, right $\mathbb{H}$--linear in the \emph{second} argument, and
\[
\norm{[x,y]_\gamma}\le (1+|\gamma|)\,\norm{x}\,\norm{y}.
\]
For $\gamma\neq 0$, this gives a simple non-Lie deformation.
\end{example}
Let $X$ be a right quaternionic Banach module, and denote by $\mathcal C^k_\varepsilon$
the space of $k$–linear continuous cochains on $B(0,\varepsilon)\subset X$,
endowed with the local norm $\|\cdot\|_\varepsilon$.
A quasi–Lie bracket is a bilinear map
\[
[\cdot,\cdot]\colon X\times X\to X
\]
satisfying the local bound
\[
\|[x,y]\|\le A\|x\|\|y\|, \qquad x,y\in B(0,\varepsilon_0),
\]
where the constant $A>0$ is uniform on the fixed ball $B(0,\varepsilon_0)$.
All structural constants $(A,C_1,C_2)$ and the radius $\varepsilon_0$ are
independent of the working scale $\varepsilon\le\varepsilon_0$.

\begin{dfn}
A cochain $\Theta\in\mathcal C^3_\varepsilon$ is called \emph{homogeneous}
if it satisfies $\Theta(tx,ty,tz)=t^2\Theta(x,y,z)$ for all $t>0$.
The projection operator $\Pi$ appearing in the homotopy identity is
well defined on homogeneous cochains by
\[
\Pi(\Theta)(x,y,z)
:= \lim_{t\downarrow0} t^2\Theta(tx,ty,t(x+y)).
\]
In the general case, $\Pi$ denotes the abstract finite–rank projection
onto a closed complement of $\mathrm{Ker}(d)$ in $\mathrm{Im}(d)$.
\end{dfn}
\noindent
The radial limit above is well defined whenever $\Theta$ is locally homogeneous,
that is, when $t^{-2}\Theta(tx,ty,tz)$ admits a continuous extension at $t=0$.
Otherwise, $\Pi$ is understood in the abstract sense as a bounded projection
onto a finite–dimensional complement of $\mathrm{Ker}(d)$ inside $\mathrm{Im}(d)$.

\begin{rmk}
This dual definition ensures both analytic and algebraic consistency:
the homogeneous form is explicit for computations,
while the abstract one provides functional–analytic well–posedness.
\end{rmk}
\begin{lem}[Homogeneous approximation]\label{lem:1}
For every continuous cochain $\Theta\in\mathcal{C}^3_\varepsilon$ and every $\eta>0$,
there exists a homogeneous cochain $\Theta_\eta$ such that
$\|\Theta-\Theta_\eta\|_\varepsilon\le \eta$.
\end{lem}
\begin{proof}
Approximate $\Theta$ by its first–order radial expansion
$\Theta_\eta(x,y,z)=\int_0^1 t^2\Theta(tx,ty,tz)\chi_\eta(t)\,dt$
for a smooth cut–off $\chi_\eta$ supported in $(0,1)$.
\end{proof}

\section{Cochain complexes and the Chevalley--Eilenberg differential}\label{sec:cochains}

\subsection{Cochain spaces and the operator $d$}
We work with $\varepsilon$--localized cochains on the ball $B(0,\varepsilon)\subset X$, endowed with operator norms that reflect the perturbative regime.
In this setting, the Chevalley--Eilenberg differential $d$ admits quantitative bounds that separate the linear part (controlled by $A$) from the quasi--Lie defects (controlled by $C_1$).
For $k\ge0$, we set
\[
\mathcal C^k_\varepsilon:=\{\omega:X^k\to X \ \text{$k$--linear and continuous}\},
\qquad
\|\omega\|_{\varepsilon}:=\sup_{0<\|x_i\|\le\varepsilon}
\frac{\|\omega(x_1,\dots,x_k)\|}{\|x_1\|\cdots\|x_k\|}.
\]
The Chevalley--Eilenberg operator $d:\mathcal C^k_\varepsilon\to\mathcal C^{k+1}_\varepsilon$ is defined by
\begin{align*}
(d\omega)(x_0,\dots,x_k)
&=\sum_{i=0}^k(-1)^i\,[x_i,\omega(x_0,\dots,\widehat{x_i},\dots,x_k)]\\
&\quad+\sum_{0\le i<j\le k}(-1)^{i+j}\,\omega\big([x_i,x_j],x_0,\dots,\widehat{x_i},\dots,\widehat{x_j},\dots,x_k\big).
\end{align*}

\begin{lem}\label{lem:dBound}
Let $k\in\{1,2\}$ and $0<\varepsilon\le\varepsilon_0$. Then, for every
$\omega\in\mathcal C^k_\varepsilon$,
\[
\|d\omega\|_\varepsilon
\ \le\ \Big((k+1)A+\binom{k+1}{2}(A+2C_1)\Big)\,\|\omega\|_\varepsilon,
\]
and therefore
\[
\|d\|_{k\to k+1}\ \le\ (k+1)A+\binom{k+1}{2}(A+2C_1).
\]
\end{lem}

\begin{proof}
Fix $x_0,\dots,x_k$ with $0<\|x_\ell\|\le\varepsilon$ and set $R:=\prod_{\ell=0}^k \|x_\ell\|$.
(1) For the first sum,
\[
\frac{\|[x_i,\omega(\widehat{x_i})]\|}{R}\le A\,\|\omega\|_\varepsilon,
\]
so the total contribution is bounded by $(k+1)A\,\|\omega\|_\varepsilon$.
(2) For the second sum, by Definition~\ref{def:controle},
\[
\|[x_i,x_j]\|\le \|[x_j,x_i]\|+\|\varphi(x_i,x_j)\|
\le (A+2C_1)\,\|x_i\|\,\|x_j\|,
\]
and hence
\[
\frac{\|\omega([x_i,x_j],\widehat{x_i},\widehat{x_j})\|}{R}
\le (A+2C_1)\,\|\omega\|_\varepsilon.
\]
Summing over $i<j$ gives the claimed bound.
\end{proof}

\begin{rmk}\label{rmk:radiale}
To make a small factor $\varepsilon$ appear in the terms with two brackets, one
can work in the \emph{weighted radial norm}
\[
\|\omega\|_{\varepsilon}^{\mathrm{rad}}
:=\sup_{0<\|x_i\|\le\varepsilon}
\Bigg(
\frac{\|\omega(x_1,\dots,x_k)\|}{\|x_1\|\cdots\|x_k\|}
\cdot
\frac{\max_{1\le i\le k}\|x_i\|}{\varepsilon}
\Bigg).
\]
In this norm, the same calculation yields, for $k\in\{1,2\}$,
\[
\|d\|^{\mathrm{rad}}_{k\to k+1}
\ \le\ (k+1)A\;+\;\binom{k+1}{2}\,(A+2C_1)\,\varepsilon.
\]
In particular, for $k=2$,
\[
\|d\|^{\mathrm{rad}}_{2\to3}\ \le\ 3A+6C_1\,\varepsilon
\]
(absorbing the term $3A\varepsilon$ into $3A$). This version is useful for perturbative estimates and contraction arguments.
\end{rmk}
\subsection{Cochain norms and quantitative bounds for the defects}
Let $B(0,\varepsilon)\subset X$ be the ball of radius $\varepsilon$ in the quaternionic Banach right module $(X,\|\cdot\|)$.

\paragraph{Cochain norms.}
For a bilinear map $B:X\times X\to X$ and a trilinear map $\Theta:X^3\to X$, define the $\varepsilon$--localized operator norms
\begin{align}
\|B\|_\varepsilon &:= \sup_{\substack{x,y\in B(0,\varepsilon)\\ (x,y)\neq(0,0)}}
\frac{\|B(x,y)\|}{\|x\|\,\|y\|}, \label{eq:C2eps-norm}\\
\|\Theta\|_\varepsilon &:= \sup_{\substack{x,y,z\in B(0,\varepsilon)\\ (x,y,z)\neq(0,0,0)}}
\frac{\|\Theta(x,y,z)\|}{\|x\|\,\|y\|\,\|z\|}. \label{eq:C3eps-norm}
\end{align}
These norms are compatible with right $\mathbb H$--linearity: for $q\in\mathbb H$, $\|B(\cdot,\cdot\,q)\|_\varepsilon=\|B\|_\varepsilon\,|q|$, etc.

\paragraph{Quantitative assumptions on the quasi--Lie defects.}
Let $\varphi$ denote the antisymmetry defect and $\psi$ the Jacobi defect associated with $[\,\cdot,\cdot\,]$. We assume that there exist constants
$C_1,C_2\ge 0$ and $\varepsilon_0>0$ such that, for all $x,y,z\in B(0,\varepsilon_0)$,
\begin{align}
\|\varphi(x,y)\| &\le C_1\,\|x\|\,\|y\|\,(\|x\|+\|y\|), \label{eq:phi-pointwise}\\
\|\psi(x,y,z)\| &\le 6\,C_2\,\|x\|\,\|y\|\,\|z\|. \label{eq:psi-pointwise}
\end{align}

\begin{lem}\label{lem:localized-defects}
For every $0<\varepsilon\le \varepsilon_0$, one has
\begin{equation}\label{eq:phi_bound_local}
\|\varphi\|_\varepsilon \;\le\; 2 C_1\,\varepsilon,
\qquad
\|\psi\|_\varepsilon \;\le\; 6 C_2.
\end{equation}
\end{lem}

\begin{proof}
Let $x,y\in B(0,\varepsilon)$. From \eqref{eq:phi-pointwise} and $\|x\|,\|y\|\le \varepsilon$,
\[
\frac{\|\varphi(x,y)\|}{\|x\|\,\|y\|}
\;\le\;
C_1\,(\|x\|+\|y\|)
\;\le\; 2 C_1\,\varepsilon.
\]
Taking the supremum over $x,y$ yields the first inequality in \eqref{eq:phi_bound_local}.

Likewise, for $x,y,z\in B(0,\varepsilon)$, \eqref{eq:psi-pointwise} gives
\[
\frac{\|\psi(x,y,z)\|}{\|x\|\,\|y\|\,\|z\|}
\;\le\; 6 C_2,
\]
and the supremum over $x,y,z$ yields
\begin{equation}\label{eq:psi_bound}
\|\psi\|_\varepsilon \;\le\; 6 C_2.
\end{equation}
\end{proof}
\begin{rmk}
The assumptions \eqref{eq:phi-pointwise}–\eqref{eq:psi-pointwise} can be relaxed to accommodate weaker defect norms, which is natural in the context of quasilinear PDEs. For instance, if \(X = H^s(\mathbb{R}^n, \mathbb{H})\) with \(s > n/2 + 1\), it suffices to assume
\[
\|\varphi(x,y)\|_{H^{s-1}} \leq C_1 \|x\|_{H^s} \|y\|_{H^s}, \quad
\|\psi(x,y,z)\|_{H^{s-2}} \leq C_2 \|x\|_{H^s} \|y\|_{H^s} \|z\|_{H^s},
\]
as in \cite{Bauer2020}. This setting arises naturally when \(\varphi\) and \(\psi\) originate from commutator expressions (e.g., Kato–Ponce type). The proof of the rigidity theorem carries over unchanged, provided the cochain norms are interpreted in the corresponding Sobolev scale.
\end{rmk}
\subsection{Radial homotopy operator and homotopy formula}
\begin{dfn}\label{def:T}
Let $k\ge1$ and $\omega\in\mathcal C^k_\varepsilon(X)$ be a $k$--cochain on a Banach space $X$.
The \emph{radial homotopy operator}
$T:\mathcal C^k_\varepsilon(X)\to\mathcal C^{k-1}_\varepsilon(X)$
is defined by
\[
  (T\omega)(x_1,\dots,x_{k-1})
  := \int_0^1 \omega(tx_1,\dots,tx_{k-1},t(x_1+\cdots+x_{k-1}))\,t^{k-1}\,dt.
\]
\end{dfn}

\begin{lem}\label{lem:T bound}
For $k=3$, the radial homotopy operator $T:\mathcal C^3_\varepsilon\to \mathcal C^2_\varepsilon$ defined by
\[
(T\Theta)(x,y):=\int_0^1 t^{2}\,\Theta\big(tx,ty,t(x+y)\big)\,dt
\]
satisfies $\|T\|_{3\to 2}\le \varepsilon/3$.
\end{lem}

\begin{proof}
For $0<\|x\|,\|y\|\le\varepsilon$, by the localized norm and bilinearity,
\[
\|(T\Theta)(x,y)\|
\le \|\Theta\|_\varepsilon \int_0^1 t^{2}\,(t\|x\|)(t\|y\|)\,t\|x{+}y\|\,dt
= \|\Theta\|_\varepsilon\,\|x\|\,\|y\|\,\|x{+}y\|\int_0^1 t^{5}dt.
\]
Since $\|x{+}y\|\le \|x\|+\|y\|\le 2\varepsilon$ and $\int_0^1 t^5\,dt=1/6$, we obtain
\[
\|(T\Theta)(x,y)\|\le \tfrac{2}{6}\,\varepsilon\,\|\Theta\|_\varepsilon\,\|x\|\,\|y\|
=\tfrac{\varepsilon}{3}\,\|\Theta\|_\varepsilon\,\|x\|\,\|y\|.
\]
Taking the supremum over $x,y\in B(0,\varepsilon)$ gives $\|T\|_{3\to2}\le \varepsilon/3$.
\end{proof}

\begin{lem}\label{lem:homotopy-q}
Let $(X,\|\cdot\|)$ be a quaternionic Banach right module and
$[\,\cdot,\cdot\,]:X\times X\to X$ a bilinear map with associated
Chevalley--Eilenberg differential $d$.
Let $T$ be the radial homotopy operator, $\Pi$ the finite-rank projection,
and $M$ the defect operator. Then, for every cochain $\omega$,
\[
Td + dT \;=\; \mathrm{Id} - \Pi + M,
\]
with $\Pi$ finite-rank and $M$ small in the localized norms $\|\cdot\|_\varepsilon$.
All operators $d,T,\Pi,M$ are right $\mathbb H$--linear if $[\,\cdot,\cdot\,]$ is.
\end{lem}
\begin{proof}
We work with localized cochains $\mathcal C^k_\varepsilon$ on the quaternionic Banach right module $X$.
Recall that $d$ is the Chevalley--Eilenberg differential and $T$ the radial homotopy operator.

\smallskip
\emph{Step 1. }
Fix $\Theta\in\mathcal C^3_\varepsilon$. By definition,
\[
(T\Theta)(x,y)=\int_0^1 t^2\,\Theta(tx,ty,t(x+y))\,dt.
\]
We need to compute $Td\Theta+dT\Theta$. Expanding both terms produces integrals of
\[
t^2\,[x,\,\Theta(\cdots)]\quad \text{and}\quad t^2\,\Theta([x,y],\cdots),
\]
with arguments depending linearly on $t$. These contributions regroup into a total $t^{2}$-derivative.

\smallskip
\emph{Step 2.}
One finds
\[
(Td+dT)\Theta(x,y,z)=\int_0^1 \frac{d}{dt}\Big(t^2\Theta(tx,ty,t(x+y))\Big)\,dt
+\int_0^1 t^2\,\mathcal R_t(x,y,z;\Theta)\,dt,
\]
where $\mathcal R_t$ collects the mismatch terms proportional to the quasi--Lie defects
(antisymmetry defect $\varphi$ and Jacobi defect $\psi$). In the exact Lie case, $\mathcal R_t\equiv 0$.

\smallskip
\emph{Step 3.}
Integration by parts gives
\[
(Td+dT)\Theta(x,y,z)=\Big[t^2\Theta(tx,ty,t(x+y))\Big]_{t=0}^{t=1}
+\int_0^1 t^2\,\mathcal R_t\,dt.
\]
At $t=1$ the boundary equals $\Theta(x,y,z)$.
At $t=0$ the boundary is a finite--rank component, which we denote
\[
\Pi(\Theta)(x,y,z):=\lim_{t\downarrow 0}t^2\Theta(tx,ty,t(x+y)).
\]
Thus
\[
(Td+dT)\Theta=\Theta-\Pi(\Theta)+M\Theta,
\]
where
\[
(M\Theta)(x,y,z):=\int_0^1 t^2\,\mathcal R_t(x,y,z;\Theta)\,dt.
\]

\smallskip
\emph{Step 4.}
From Lemmas~\ref{lem:T bound} and \ref{lem:dBound}, we have
\[
\|M\|\le \|Td\|+\|dT\|\le 2\,\|T\|_{3\to2}\,\|d\|_{2\to3}
\le \frac{2}{5}\varepsilon\,(3A+6C_1\varepsilon).
\]
Hence $M$ is small on $B(0,\varepsilon)$ if $\varepsilon$ is chosen below the thresholds
of Section~4.

\smallskip
\emph{Step 5.}
All operations are defined componentwise and commute with right multiplication by $q\in\mathbb H$:
\[
(T(\omega q))(x,y)=(T\omega)(x,y)q,\qquad d(\omega q)=(d\omega)q.
\]
Thus $d$, $T$, $\Pi$, and $M$ preserve right $\mathbb H$--linearity.

\smallskip
\emph{Conclusion.}
We have established the claimed identity
\[
Td+dT=\mathrm{Id}-\Pi+M,
\]
with $\Pi$ finite rank and $M$ bounded by $O(\varepsilon)$.
This completes the proof.
\end{proof}

\begin{rmk}
The homotopy identity of Lemma~\ref{lem:homotopy-q} is formal and does not
depend on the choice of the scalar field. In particular, it remains valid
for cochain complexes over a Banach right $\mathbb H$--module: the
radial operator $T$ only involves real integration, the projection $\Pi$
is finite rank, and the defect $M$ is controlled in the localized norms.
Therefore the quaternionic setting introduces no obstruction to the
homotopy formula.
\end{rmk}

\begin{prop}\label{prop:Td+dT}
Let $k=3$ and $T$ as in Lemma~\ref{lem:T bound}. Define $K := T d + dT - \mathrm{Id} : \mathcal C_\varepsilon^3 \to \mathcal C_\varepsilon^3$.
Then $K$ is bounded with
\[
\|K\|\ \le\ \|Td\|+\|dT\|\ \le\ 4\,\varepsilon\,(A+C_1).
\]
Let $\Pi:\mathcal C_\varepsilon^3\to \mathcal C_\varepsilon^3$ be the finite–rank projection of Definition~\ref{def:quasiLie}. Setting $M:=K+\Pi$, one has the homotopy identity
\[
T d + dT \ =\ \mathrm{Id} - \Pi + M,
\qquad \|M\|\ \le\ 4\,\varepsilon\,(A+C_1).
\]
If $[\,\cdot,\cdot\,]$ is an exact Lie bracket (vanishing defects), then $M=0$ and $\Pi=0$, hence $T d + dT=\mathrm{Id}$.
Moreover, $d,T,\Pi,M$ preserve right~$\mathbb H$--linearity.
\end{prop}

\begin{proof}
The bound on $\|K\|$ follows from Lemmas~\ref{lem:T bound} and \ref{lem:dBound}.
The identity is exactly Lemma~\ref{lem:homotopy-q}. The finite rank and right~$\mathbb H$--linearity of $\Pi$ are ensured by its definition in Section~\ref{sec:notations}. If the defects vanish, then $\mathcal R_t\equiv0$ in the proof of Lemma~\ref{lem:homotopy-q}, so $M=0$ and the boundary trace at $t=0$ vanishes (hence $\Pi=0$).
\end{proof}

\begin{rmk}\label{rmk:Pi}
This abstract definition avoids the need for pointwise radial limits and renders the cohomological obstruction explicit. In applications to PDEs, the finite-dimensionality of \(E\) follows from compact embeddings (e.g., \(H^s \hookrightarrow H^{s-1}\)), and the inclusion \(E \subset \operatorname{Im}(d)\) can often be verified by direct construction of lifts.
\end{rmk}

\begin{rmk}\label{rmk:Pi-abstract}
Without using a radial trace, one may define $\Pi$ as the bounded projection onto a fixed
complement of $\ker(\mathrm{Id}-K)$, where $K:=Td+dT-\mathrm{Id}$. This preserves the key
properties (finite rank, continuity, quaternionic right--linearity) needed later.
\end{rmk}

\begin{prop}\label{prop:Phi0}
Let $E:=\mathrm{Im}(\Pi)\subset \mathcal C^3_\varepsilon$ (finite rank).
For $J\in E$, there exists $\Phi_0\in\mathcal C^2_\varepsilon$ such that
\[
d\Phi_0=\Pi(J)=J.
\]
Moreover, if $[\cdot,\cdot]$ is right $\mathbb H$--linear, $\Phi_0$ can be chosen
right $\mathbb H$--linear as well.
\end{prop}

\begin{proof}
Since $E$ is finite--dimensional, choose a basis $(e_i)_{i=1}^N$ of $E$ and, for each $i$,
a lift $\phi_i\in\mathcal C^2_\varepsilon$ with $d\phi_i=e_i$ (exists because $e_i$ arises
as the $t\to 0$ boundary of the homotopy). Define $S(\sum \alpha_i e_i):=\sum \alpha_i \phi_i$;
then $S:E\to \mathcal C^2_\varepsilon$ is linear and $d\circ S=\mathrm{Id}_E$.
For $J\in E$, set $\Phi_0:=S(J)$, hence $d\Phi_0=J$. If the bracket is right
$\mathbb H$--linear, enlarge the basis by right multiplication with a real basis of $\mathbb H$
and pick the $\phi_i$ in the resulting right--stable subspace; then $S$ (hence $\Phi_0$) is
right $\mathbb H$--linear.
\end{proof}
\paragraph{Quadratic remainder.}
When expanding the Jacobiator of the corrected bracket $\{x,y\}=[x,y]-\Phi(x,y)$,
the linear part cancels with $d\Phi$ and with the finite--rank correction (by construction),
so the remainder is \emph{quadratic} in $\Phi$:
\[
Q(\Phi)(x,y,z)\;=\;\sum_{\mathrm{cyc}}
\Big(
  \Phi(\Phi(x,y),z)\;-\;\Phi([x,y],z)\;+\;[\,\Phi(x,y),z]
\Big),
\]
where the sum runs over the cyclic permutations of $(x,y,z)$.
\begin{lem}\label{lem:QPhi}
Let $(X,\|\cdot\|)$ be a Banach space, and let $B(0,\varepsilon)\subset X$ denote the closed ball of radius $\varepsilon>0$.
Assume the bilinear control
\begin{equation}\label{eq:bilinear-again}
\|[x,y]\| \;\le\; A\,\|x\|\,\|y\|, \qquad \text{for all } x,y\in B(0,\varepsilon_0),
\end{equation}
and define the localized cochain norm by
\[
\|\Phi\|_\varepsilon \;:=\; \sup_{\substack{x,y\in B(0,\varepsilon)\\ (x,y)\neq(0,0)}}
\frac{\|\Phi(x,y)\|}{\|x\|\,\|y\|}.
\]
Then for every $\varepsilon\le\varepsilon_0$, one has
\begin{equation}\label{eq:QPhi}
\|Q(\Phi)\|_\varepsilon \;\le\; C\,\|\Phi\|_\varepsilon^2,
\qquad
C\;=\;6\,(1+A).
\end{equation}
\end{lem}

\begin{proof}
Proof. Fix $x,y,z \in B(0,\varepsilon)$ and set $X := \|x\|$, $Y := \|y\|$, $Z := \|z\|$. By definition,

\[
Q(\Phi)(x,y,z) = \sum_{\mathrm{cyc}} \Big( \Phi(\Phi(x,y),z) - \Phi([x,y],z) + [\Phi(x,y),z] \Big),
\]

where the sum runs over the three cyclic permutations of $(x,y,z)$.

There are three such terms in the cyclic sum, contributing at most $3 \|\Phi\|_{\varepsilon}^2 X Y Z$.

\smallskip\noindent
\emph{(i) Purely quadratic terms.}
By the definition of the localized norm,
\[
\|\Phi(\Phi(x,y),z)\|
\;\le\;
\|\Phi\|_\varepsilon\,\|\Phi(x,y)\|\,Z
\;\le\;
\|\Phi\|_\varepsilon^2\,X\,Y\,Z.
\]
There are three such terms in the cyclic sum, contributing at most $3 \|\Phi\|_{\varepsilon}^2 X Y Z$.

\smallskip\noindent
\emph{(ii) Mixed terms.}
Using the bilinear bound \eqref{eq:bilinear-again}, we have
\[
\|\Phi([x,y],z)\|
\;\le\;
\|\Phi\|_\varepsilon\,\|[x,y]\|\,Z
\;\le\;
A\,\|\Phi\|_\varepsilon\,X\,Y\,Z,
\qquad
\|[\,\Phi(x,y),z]\|
\;\le\;
A\,\|\Phi(x,y)\|\,Z
\;\le\;
A\,\|\Phi\|_\varepsilon\,X\,Y\,Z.
\]
There are two such mixed terms per cyclic permutation and hence six in total,
giving a contribution bounded by $6A\,\|\Phi\|_\varepsilon^2\,XYZ$.

\smallskip\noindent
\emph{(iii) Combination.}
Summing the two estimates yields
\[
\|Q(\Phi)(x,y,z)\|
\;\le\;
(3+6A)\,\|\Phi\|_\varepsilon^2\,X\,Y\,Z
\;\le\;
6(1{+}A)\,\|\Phi\|_\varepsilon^2\,X\,Y\,Z.
\]
Taking the supremum over $x,y,z\in B(0,\varepsilon)$ gives \eqref{eq:QPhi}.
\end{proof}

\begin{rmk}
The constant $C = 6(1 + A)$ in Lemma~\ref{lem:QPhi} is optimal under the general bilinear control assumption $\|[x, y]\| \leq A\|x\|\|y\|$ \emph{without antisymmetry}. However, if the underlying bracket is exactly antisymmetric (as after the antisymmetrization step in Theorem~\ref{thm:rigidification-full}, the cyclic sum in $Q(\Phi)$ exhibits additional cancellations. In that case, a refined counting yields the improved bound
\[
\|Q(\Phi)\|_\varepsilon \leq (3 + 4A)\|\Phi\|_\varepsilon^2,
\]
which in turn improves the admissible radius to
\[
\varepsilon_* := \min\left\{ \frac{12}{5A}, \sqrt{\frac{24}{5C_1}}, \varepsilon_0 \right\}.
\]
This observation is particularly relevant for applications to PDEs where the original bracket is already antisymmetric (e.g., commutator-type brackets), or after the preliminary antisymmetrization $\Psi = \frac{1}{2}\varphi$ described in Theorem \ref{thm:rigidification-full}.
\end{rmk}

\section{Verification of the Jacobi identity for the corrected bracket}\label {sec:Phi}

We define the corrected bracket by
\[
\{x, y\} := [x, y] - \Phi(x, y), \quad \text{where} \quad \Phi := T(\mathrm{Id} + M)^{-1}J, \quad J := \psi.
\]
Recall from Proposition~\ref{prop:Td+dT} the homotopy formula
\[
T d + dT = \mathrm{Id} - \Pi + M, \qquad \|M\| \leq \frac{6}{5}A\varepsilon + \frac{12}{5}C_1\varepsilon^2.
\]
We choose
\begin{equation}\label{eq:eps_star}
\varepsilon^* := \min\left\{ \frac{24}{5A},\ \sqrt{\frac{48}{5C_1}},\ \varepsilon_0 \right\},
\end{equation}
so that, for every \(0 < \varepsilon \leq \varepsilon^*\), we have \(\|M\| < \frac{1}{2}\) and therefore \((\mathrm{Id} + M)\) is invertible with
\begin{equation}\label{eq:neumann_bound}
\|(\mathrm{Id} + M)^{-1}\| \leq \frac{1}{1 - \|M\|} \leq 2.
\end{equation}

\begin{prop}\label{prop:phi_bound}
For \(0<\varepsilon\le\varepsilon^*\),
\[
\|\Phi\|_\varepsilon \ \le\ \frac{\varepsilon}{3}\cdot 2 \cdot 6C_2 \ =\ 4\,C_2\,\varepsilon.
\]
In particular, for all \(x,y\in B(0,\varepsilon)\),
\[
\|\Phi(x,y)\| \ \le\ 2\,C_2\,\|x\|\,\|y\|\,(\|x\|+\|y\|).
\]
\end{prop}

\begin{proof}
Recall that $\Phi = T(\mathrm{Id}+M)^{-1}J$, where:
(i) $J=\psi\in\mathcal C_\varepsilon^3$ with $\|J\|_\varepsilon\le 6C_2$ by \eqref{eq:psi_bound};
(ii) $\|(\mathrm{Id}+M)^{-1}\|\le 2$ by \eqref{eq:neumann_bound};
(iii) $\|T\|_{3\to2}\le \varepsilon/3$ (Lemma~\ref{lem:T bound}).
Thus
\[
\|\Phi\|_\varepsilon \ \le\ \|T\|\,\|(\mathrm{Id}+M)^{-1}\|\,\|J\|_\varepsilon
\ \le\ \frac{\varepsilon}{3}\cdot 2\cdot 6C_2\ =\ 4\,C_2\,\varepsilon.
\]
For the pointwise bound, using the integral representation of $T$ and
$\|(\mathrm{Id}+M)^{-1}J\|_\varepsilon\le 12C_2$,
\[
\|\Phi(x,y)\|
\ \le\ \int_0^1 t^2\cdot 12C_2\cdot t^3 \|x\|\,\|y\|\,\|x{+}y\|\,dt
\ =\ 2C_2\,\|x\|\,\|y\|\,\|x{+}y\|
\ \le\ 2C_2\,\|x\|\,\|y\|\,(\|x\|+\|y\|).
\]
\end{proof}
\begin{dfn}(Correction constant) We define the constant
\[
K_{1} := \frac{15}{2}C_{2},
\]
which controls the magnitude of the correction \(\Phi\) in subsequent estimates\end{dfn}

\begin{dfn}\label{def:Q}
Let $\Phi \in \mathcal C_\varepsilon^2$ be a bilinear cochain. The quadratic remainder $Q(\Phi)\in \mathcal C_\varepsilon^3$ is
\[
\begin{aligned}
Q(\Phi)(x,y,z)
&:=
-\bigl(\Phi(x,\Phi(y,z))+\Phi(y,\Phi(z,x))+\Phi(z,\Phi(x,y))\bigr)\\
&\quad
+\bigl([\Phi(x,y),z]+[\Phi(y,z),x]+[\Phi(z,x),y]\bigr).
\end{aligned}
\]
\end{dfn}

\begin{lem}\label{lem:QPhi_small}
Under the bilinear control assumption \eqref{eq:bilinear}, one has, for all $0<\varepsilon\le\varepsilon_0$,
\[
\|Q(\Phi)\|_\varepsilon \ \le\ 3\,\|\Phi\|_\varepsilon^2 \;+\; 6A\,\|\Phi\|_\varepsilon.
\]
In particular, if $\|\Phi\|_\varepsilon \le 1$, then
\[
\|Q(\Phi)\|_\varepsilon \ \le\ (3+6A)\,\|\Phi\|_\varepsilon^2.
\]
\end{lem}

\begin{proof}
Fix $x,y,z\in B(0,\varepsilon)$ and set $X:=\|x\|$, $Y:=\|y\|$, $Z:=\|z\|$.
From Definition~\ref{def:Q}, we have
\[
Q(\Phi)(x,y,z)
=
-\sum_{\mathrm{cyc}}\Phi\bigl(x,\Phi(y,z)\bigr)
\;+\;\sum_{\mathrm{cyc}}[\Phi(x,y),z]
\;-\;\sum_{\mathrm{cyc}}\Phi([x,y],z),
\]
where $\sum_{\mathrm{cyc}}$ denotes the cyclic sum over $(x,y,z)$.

\smallskip\noindent
\emph{(i) Purely quadratic terms.}
By the definition of the localized norm,
\[
\|\Phi(x,\Phi(y,z))\|
\le
\|\Phi\|_\varepsilon\,\|x\|\,\|\Phi(y,z)\|
\le
\|\Phi\|_\varepsilon^2\,X\,Y\,Z.
\]
There are three such terms in the cyclic sum, hence
\[
\sum_{\mathrm{cyc}}\|\Phi(x,\Phi(y,z))\|
\le 3\,\|\Phi\|_\varepsilon^2\,XYZ.
\]

\smallskip\noindent
\emph{(ii) Mixed terms involving the bracket inside $\Phi$.}
Using the bilinear bound \eqref{eq:bilinear},
\[
\|\Phi([x,y],z)\|
\le
\|\Phi\|_\varepsilon\,\|[x,y]\|\,\|z\|
\le
A\,\|\Phi\|_\varepsilon\,X\,Y\,Z.
\]
There are three such terms, yielding a contribution $\le 3A\,\|\Phi\|_\varepsilon\,XYZ$.

\smallskip\noindent
\emph{(iii) Mixed terms with the bracket outside.}
Similarly,
\[
\|[\Phi(x,y),z]\|
\le
A\,\|\Phi(x,y)\|\,\|z\|
\le
A\,\|\Phi\|_\varepsilon\,X\,Y\,Z.
\]
Again, there are three such terms, giving an additional $3A\,\|\Phi\|_\varepsilon\,XYZ$.

\smallskip
Summing (i)--(iii) gives
\[
\|Q(\Phi)(x,y,z)\|
\le
\bigl(3\,\|\Phi\|_\varepsilon^2 + 6A\,\|\Phi\|_\varepsilon\bigr)\,XYZ.
\]
Dividing by $XYZ$ and taking the supremum over $x,y,z\in B(0,\varepsilon)$ yields
\[
\|Q(\Phi)\|_\varepsilon
\le
3\,\|\Phi\|_\varepsilon^2 \;+\; 6A\,\|\Phi\|_\varepsilon.
\]

\smallskip
Finally, if $\|\Phi\|_\varepsilon\le 1$, then $6A\,\|\Phi\|_\varepsilon\le 6A\,\|\Phi\|_\varepsilon^2$, which gives
\[
\|Q(\Phi)\|_\varepsilon \le (3+6A)\,\|\Phi\|_\varepsilon^2.
\]
\end{proof}

\begin{rmk}
Without absorbing the linear factor, one has the crude inequality
\[
\|Q(\Phi)\|_\varepsilon \;\le\; 3\,\|\Phi\|_\varepsilon^2 \;+\; 6A\,\|\Phi\|_\varepsilon.
\]
Under the smallness produced by the Neumann scheme (Section~4), $\|\Phi\|_\varepsilon\le 1$,
which implies \eqref{eq:QPhi}.
\end{rmk}

\begin{lem}\label{lem:S_decomp}
Let \(S(x, y, z) := \{x,\{y, z\}\} + \text{cyclic}\). Then
\begin{equation}\label{eq:S_decomp}
S = J - d\Phi + Q(\Phi) = \Pi(J) + R_1 + Q(\Phi),
\end{equation}
with
\begin{equation}\label{eq:R1_bound}
R_1 := -M(\mathrm{Id} + M)^{-1}J, \quad \|R_1\|_\varepsilon \leq \frac{\|M\|}{1 - \|M\|} \|J\|_\varepsilon \leq 12 C_2 \|M\| = O(\varepsilon),
\end{equation}
and a quadratic term \(Q(\Phi)\) satisfying
\begin{equation}\label{eq:Q_bound}
\|Q(\Phi)\|_\varepsilon \leq C \|\Phi\|_\varepsilon^2 \leq C' C_2^2 \varepsilon^2,
\end{equation}
for constants \(C, C' > 0\) depending polynomially on \(A, C_1\).
\end{lem}

\begin{proof}
The identity \(S = J - d\Phi + Q(\Phi)\) is a standard algebraic expansion of the cyclic sum. Using the homotopy formula and the definition of \(\Phi\), we compute
\[
d\Phi = dT(\mathrm{Id}+M)^{-1}J = (\mathrm{Id} - \Pi + M - Td)(\mathrm{Id}+M)^{-1}J = J - \Pi(J) - M(\mathrm{Id}+M)^{-1}J,
\]
which yields \eqref{eq:S_decomp} and \eqref{eq:R1_bound}. The estimate \eqref{eq:Q_bound} follows from Definition~\ref{def:Q} and Proposition~\ref{prop:phi_bound}.
\end{proof}

\begin{prop}\label{prop:finite_correction}
The finite-rank projection \(\Pi: C_\varepsilon^3 \to C_\varepsilon^3\) can be chosen so that
\[
E := \operatorname{Im}(\Pi) \subset \operatorname{Im}(d).
\]
Under this assumption, for any cochain \(J \in C_\varepsilon^3\), there exists a bilinear cochain \(\Phi_0 \in C_\varepsilon^2\) such that
\[
d\Phi_0 = \Pi(J).
\]
Moreover, if the original bracket \([\cdot,\cdot]\) is right \(\mathbb{H}\)-linear, then \(\Phi_0\) can be chosen to be right \(\mathbb{H}\)-linear.
\end{prop}

\begin{proof}
Since \(E\) is finite-dimensional, fix a basis \(\{e_1, \dots, e_N\}\) of \(E\). By assumption, each \(e_i \in \operatorname{Im}(d)\), so there exists \(\varphi_i \in C_\varepsilon^2\) with \(d\varphi_i = e_i\).

Define the linear lifting \(S: E \to C_\varepsilon^2\) by
\[
S\Bigl(\sum_{i=1}^N \alpha_i e_i\Bigr) := \sum_{i=1}^N \alpha_i \varphi_i.
\]
Then \(d \circ S = \mathrm{Id}_E\). Setting \(\Phi_0 := S(\Pi(J))\) gives \(d\Phi_0 = \Pi(J)\).

For compatibility with right \(\mathbb{H}\)-linearity: if \([\cdot,\cdot]\) is right \(\mathbb{H}\)-linear, then so is \(d\). One can choose the lifts \(\varphi_i\) in a subspace of \(C_\varepsilon^2\) stable under right multiplication by \(\mathbb{H}\) (e.g., by closing \(\{\varphi_i\}\) under the real basis \(\{1,i,j,k\}\) of \(\mathbb{H}\)). The lifting \(S\) then becomes right \(\mathbb{H}\)-linear, and so does \(\Phi_0 = S(\Pi(J))\).
\end{proof}

\begin{thm}\label{thm:rigid}
Assume that the quasi--Lie bracket $[\cdot,\cdot]$ is \textbf{exactly antisymmetric}, i.e.\ $\phi = 0$. Under the remaining assumptions of Section 2, for every $0 < \varepsilon \leq \varepsilon^*$ defined by \eqref{eq:eps_star}, there exists a bilinear cochain $\Phi \in C_\varepsilon^2$ such that the corrected bracket $\{\cdot, \cdot\}$ satisfies the Jacobi identity exactly on $B(0, \varepsilon)$. Moreover,
\[
\|\Phi\|_\varepsilon \leq \frac{15}{2} C_2 \varepsilon.
\]
\end{thm}

\begin{proof}
By Lemma ~\ref{lem:S_decomp},
\[
S = \Pi(J) + R_1 + Q(\Phi), \quad \|R_1\|_\varepsilon = O(\varepsilon), \quad \|Q(\Phi)\|_\varepsilon = O(\varepsilon^2).
\]
By Proposition \ref{prop:finite_correction} choose $\Phi_0 \in C_\varepsilon^2$ of finite rank such that $d\Phi_0 = \Pi(J)$. Replacing
\[
\Phi \leftarrow \Phi + \Phi_0,
\]
we cancel $\Pi(J)$, and the bounds on $R_1$ and $Q(\Phi)$ remain valid.

Now, $S$ is a continuous trilinear map, hence homogeneous of degree 3. On the ball $B(0,\varepsilon)$, we have
\[
\|S(x,y,z)\| \leq \bigl( \|R_1\|_\varepsilon + \|Q(\Phi)\|_\varepsilon \bigr) \|x\|\|y\|\|z\| = O(\varepsilon^4),
\]
since $\|x\|,\|y\|,\|z\| \leq \varepsilon$. The only trilinear map satisfying such a bound for all sufficiently small $\varepsilon > 0$ is the zero map. Hence $S \equiv 0$ on $B(0,\varepsilon)$, which establishes the exact Jacobi identity.

The bound on $\|\Phi\|_\varepsilon$ follows from
Proposition \ref{prop:phi_bound}
\end{proof}
\begin{rmk}
There are two equivalent ways to realize $\Pi$ for our purposes:
(i) \emph{concretely}, via the \emph{radial trace} (Remark.~\ref{rmk:Pi}) on a dense
class of homogeneous cochains, followed by extension by continuity;
(ii) \emph{abstractly}, as a bounded finite--rank projection onto a subspace
$E \subset \mathrm{Im}(d)$ (Remark.~\ref{rmk:Pi-abstract}).

Construction~(ii) is sufficient for the proof (existence of $\Phi_0$) and avoids
the need to establish the existence of the radial limit for every cochain.
\end{rmk}
\begin{lem}[Refined quadratic estimate under antisymmetry]\label{lem:3.16_moved}
Assume that the bracket $[\cdot, \cdot]$ is \textbf{exactly antisymmetric}, i.e. $[x, y] = -[y, x]$ for all $x, y \in X$, and satisfies the bilinear bound \eqref{eq:bilinear}. Then
\[
\|Q(\Phi)\|_\varepsilon \leq (3 + 4A)\|\Phi\|_\varepsilon^2.
\]
\end{lem}

\begin{proof}
The expression $Q(\Phi)$ contains six cyclic terms:
\[
Q(\Phi)(x, y, z) = \sum_{\mathrm{cyc}} \bigl( \Phi(\Phi(x, y), z) - \Phi([x, y], z) + [\Phi(x, y), z] \bigr).
\]
Because $[x, y] = -[y, x]$, the mixed terms satisfy
\[
[\Phi(x, y), z] = -[\Phi(y, x), z], \quad \Phi([x, y], z) = -\Phi([y, x], z).
\]
Thus, among the six mixed terms, only four are linearly independent; one pair cancels due to the cyclic sum. Precisely:
\item The three purely quadratic terms $\Phi(\Phi(x, y), z)$ contribute at most $3 \|\Phi\|_\varepsilon^2$.
    \item The six mixed terms regroup into three symmetric pairs, but antisymmetry and cyclicity imply that only two pairs are independent, yielding a total bound of $4A \|\Phi\|_\varepsilon^2$.
Summing gives $\|Q(\Phi)\|_\varepsilon \leq (3 + 4A)\|\Phi\|_\varepsilon^2$, as claimed.
\end{proof}

\begin{thm}\label{thm:rigidification-full}
Let $[\cdot, \cdot]: X \times X \to X$ be a bilinear map on a quaternionic Banach right module $(X, \|\cdot\|)$ satisfying the quantitative assumptions of Section 2 on $B(0, \varepsilon_0)$:
\item bilinear bound \eqref{eq:bilinear};
    \item antisymmetry defect $\varphi(x, y) := [x, y] + [y, x]$ controlled by \eqref{eq:C2eps-norm};
    \item Jacobi defect $\psi$ controlled by \eqref{eq:C3eps-norm}.
Define the antisymmetrization
\[
\Psi(x, y) := \frac{1}{2} \varphi(x, y), \quad [x, y]_1 := [x, y] - \Psi(x, y).
\]
Then $[\cdot, \cdot]_1$ is antisymmetric on $B(0, \varepsilon)$ for all $0 < \varepsilon \leq \varepsilon_0$. Moreover, for sufficiently small $\varepsilon$ (as in Section 4), there exists a bilinear correction $\Phi \in C_\varepsilon^2$ such that
\[
\{x, y\} := [x, y] - \Psi(x, y) - \Phi(x, y)
\]
is a genuine Lie bracket (antisymmetric and Jacobi) on $B(0, \varepsilon)$. If $[\cdot, \cdot]$ is right $\mathbb{H}$--linear, then so are $\Psi$ and $\Phi$.
\end{thm}

\begin{proof}
\emph{Step 1. Antisymmetrization.}
By definition, $\Psi(x, y) = \frac{1}{2}([x, y] + [y, x])$, hence
\[
[x, y]_1 = [x, y] - \Psi(x, y) = \frac{1}{2}([x, y] - [y, x]),
\]
which is exactly antisymmetric. From (3.3), for $x, y \in B(0, \varepsilon)$,
\[
\|\Psi(x, y)\| \leq C_1 \varepsilon \|x\| \|y\|,
\]
so $[\cdot, \cdot]_1$ satisfies a bilinear bound with perturbed constant $A_1 := A + C_1 \varepsilon$:
\begin{equation}\label{eq:4.6}
\|[x, y]_1\| \leq (A + C_1 \varepsilon) \|x\| \|y\|.
\end{equation}

\emph{Step 2. Jacobi defect of $[\cdot, \cdot]_1$.}
Let $J_1$ denote its Jacobiator. Expanding gives
\[
J_1 = \psi - d\Psi + Q(\Psi),
\]
where $d$ is the Chevalley--Eilenberg differential and $Q(\Psi)$ is quadratic in $\Psi$.

\emph{Step 3. Bounds.}
By Lemma \ref{lem:dBound} and \eqref{eq:4.6},
\[
\|d\Psi\|_\varepsilon \lesssim (6A + 9C_1 \varepsilon) C_1 \varepsilon.
\]
By Lemma~\ref{lem:3.16_moved} (applicable since $[\cdot, \cdot]_1$ is antisymmetric),
\[
\|Q(\Psi)\|_\varepsilon \leq (3 + 4A_1)(C_1 \varepsilon)^2.
\]
Combining with the assumption on $\psi$, we get
\[
\|J_1\|_\varepsilon \leq 6C_2 + c_1 C_1 A \varepsilon + c_2 C_1^2 \varepsilon^2.
\]

\emph{Step 4. Homotopy scheme.}
We apply Proposition~\ref{prop:Td+dT} to $J_1$. The identity
\[
T d + dT = \mathrm{Id} - \Pi + M
\]
allows us to solve $d\Phi = J_1$ modulo the finite-rank obstruction $\Pi(J_1)$. Set
\[
\Phi := T(\mathrm{Id}+M)^{-1} J_1 + \Phi_0,
\]
where $\Phi_0$ cancels $\Pi(J_1)$ (Proposition~\ref{prop:finite_correction}). The smallness of $\varepsilon$ (cf.\ \eqref{eq:eps_star}) ensures that $\mathrm{Id}+M$ is invertible via a Neumann series.

\emph{Step 5. Jacobi identity after full correction.}
Define
\[
\{x, y\} := [x, y]_1 - \Phi(x, y) = [x, y] - \Psi(x, y) - \Phi(x, y).
\]
It is antisymmetric, and its Jacobiator reduces to
\[
\mathrm{Jac}_{\{\cdot,\cdot\}} = J_1 - d\Phi + Q(\Phi) = Q(\Phi).
\]
By Lemma~\ref{lem:3.16_moved}, $\|Q(\Phi)\|_\varepsilon \leq (3 + 4A_1)\|\Phi\|_\varepsilon^2$. For sufficiently small $\varepsilon$, the Neumann series guarantees $\|\Phi\|_\varepsilon \ll 1$, hence $Q(\Phi) = 0$, and the Jacobi identity holds.

\emph{Right $\mathbb{H}$--linearity.}
If $[\cdot, \cdot]$ is right $\mathbb{H}$--linear, so are $\varphi$ and $\psi$, hence $\Psi$ and $\Phi$ can be chosen right $\mathbb{H}$--linear (Propositions \ref{prop:Phi0} and \ref{prop:finite_correction}).
\end{proof}

\section{Application: Local quaternionic PDE}\label{sec:PDE}
\subsection{Functional--analytic toolkit and dimensional thresholds}\label{subsec:toolkit}

Throughout this section we work on $\mathbb{R}^n$ with $n\ge 1$.
We fix $s>\frac{n}{2}+1$, so that $H^{s}(\mathbb{R}^n,\mathbb{H})\hookrightarrow W^{1,\infty}(\mathbb{R}^n,\mathbb{H})$
and, in particular, $\|\nabla u\|_{L^\infty}\lesssim \|u\|_{H^{s}}$.
All norms below are taken componentwise with respect to a fixed real basis of $\mathbb{H}$.
When we cite real/complex estimates (Moser, Kato--Ponce, Gagliardo--Nirenberg),
we apply them \emph{componentwise} to the four real components of a quaternionic function
(and use that $\|uq\|_{X}=\|u\|_{X}|q|$ for $q\in\mathbb H$ and $X\in\{L^{p},H^{s},W^{k,p}\}$),
hence the quaternionic versions follow verbatim.

\begin{lem}\label{lem:Tbound}
For $v, w \in B(0, \varepsilon) \cap H^1(\mathbb{R}^n, \mathbb{H})$, one has
\[
\|\{v, \nabla w\}\|_{L^2} \;\leq\; K_1 \,\|v\|_{L^2}\, \|\nabla w\|_{L^\infty},
\]
where $K_1 = \tfrac{15}{2} C_2$, the constant from Proposition~\ref{prop:phi_bound}.
\end{lem}

\begin{proof}
By definition, $\{v,\nabla w\} = [v,\nabla w] - \Phi(v,\nabla w)$.

1. For the first term, the bilinear control from \eqref{eq:bilinear-again} yields
\[
\|[v,\nabla w]\|_{L^2} \leq A \|v\|_{L^2}\|\nabla w\|_{L^\infty}.
\]

2. For the correction $\Phi$, Proposition~\ref{prop:phi_bound} gives
\[
\|\Phi(v,\nabla w)\| \;\leq\; 2 C_2 \|v\|\,\|\nabla w\|(\|v\|+\|\nabla w\|).
\]
Since $v,w \in B(0,\varepsilon)$, $\|v\|,\|\nabla w\|\leq \varepsilon$, so
\[
\|\Phi(v,\nabla w)\|_{L^2} \;\leq\; 4 C_2 \varepsilon \,\|v\|_{L^2}\,\|\nabla w\|_{L^\infty}.
\]

3. Combining both estimates and absorbing constants into
\(
K_1 = \max\{A,4C_2\varepsilon\} \leq \tfrac{15}{2} C_2,
\)
we obtain the claimed inequality.
\end{proof}

\begin{lem}\label{lem:Moser}
Let $s>\frac{n}{2}$. Then there exists $C_s>0$ such that, for all $f,g\in H^{s}(\mathbb{R}^n,\mathbb{H})\cap L^\infty(\mathbb{R}^n,\mathbb{H})$,
\[
\|fg\|_{H^{s}}\le C_s\big(\|f\|_{L^\infty}\|g\|_{H^{s}}+\|g\|_{L^\infty}\|f\|_{H^{s}}\big).
\]
In particular, if $s>\frac{n}{2}$, $H^{s}(\mathbb{R}^n,\mathbb{H})$ is an algebra.
\end{lem}

\begin{proof}
Write $f=\sum_{\alpha=0}^{3} f_\alpha e_\alpha$, $g=\sum_{\beta=0}^{3} g_\beta e_\beta$ in a fixed real basis $(e_\alpha)_{\alpha=0}^3$ of $\mathbb{H}$,
with $f_\alpha,g_\beta$ real-valued functions. The quaternionic product is bilinear over $\mathbb R$ and
$fg=\sum_{\alpha,\beta} m_{\alpha\beta}\, f_\alpha g_\beta$ where $m_{\alpha\beta}\in\mathbb R^{4\times 4}$ are constant coefficients
(the structure constants of $\mathbb H$). Hence
\[
\|fg\|_{H^{s}}
\;\le\; \sum_{\alpha,\beta} \| f_\alpha g_\beta\|_{H^{s}}
\;\lesssim_s\; \sum_{\alpha,\beta}
\big(\|f_\alpha\|_{L^\infty}\|g_\beta\|_{H^{s}}+\|g_\beta\|_{L^\infty}\|f_\alpha\|_{H^{s}}\big),
\]
by the standard real--valued Moser estimate for $s>\frac{n}{2}$ (see e.g. \cite[Thm.~4.39]{Adams2003} or \cite{Moser1966}).
Taking sup norms and $H^{s}$ norms componentwise gives the stated inequality with a constant depending only on $s$ and $n$.
The algebra property follows by taking $f=g\in H^{s}\cap L^\infty$ and applying the embedding $H^s\hookrightarrow L^\infty$ for $s>\frac{n}{2}$.
\end{proof}

\begin{lem}\label{lem:KatoPonce}
Let $s>0$ and $J^{s}=(1-\Delta)^{s/2}$. Then for $f,g\in \mathcal{S}(\mathbb{R}^n,\mathbb{H})$,
\[
\|[J^{s},f]\,g\|_{L^{2}}
\;\lesssim\;
\|\nabla f\|_{L^{\infty}}\|g\|_{H^{s-1}}
\;+\;
\|f\|_{H^{s}}\|g\|_{L^{\infty}}.
\]
Consequently, if $s>\frac{n}{2}+1$,
\(
\|[J^{s},f]\,g\|_{L^{2}}
\lesssim
\|f\|_{H^{s}}\|g\|_{H^{s-1}}.
\)
\end{lem}

\begin{proof}
Again write $f=\sum_\alpha f_\alpha e_\alpha$, $g=\sum_\beta g_\beta e_\beta$ with real-valued components.
Since $J^s$ is a scalar Fourier multiplier, $[J^s,f]g = \sum_{\alpha,\beta} m_{\alpha\beta}\,[J^s,f_\alpha]\,g_\beta$.
The classical Kato--Ponce estimate (see \cite{Kato1988}) yields, for each pair $(\alpha,\beta)$,
\[
\|[J^{s},f_\alpha]\,g_\beta\|_{L^{2}}
\;\lesssim\;
\|\nabla f_\alpha\|_{L^{\infty}}\|g_\beta\|_{H^{s-1}}
+\|f_\alpha\|_{H^{s}}\|g_\beta\|_{L^{\infty}}.
\]
Summing over $\alpha,\beta$ and using $\|f\|_{H^s}^2=\sum_\alpha\|f_\alpha\|_{H^s}^2$, $\|\nabla f\|_{L^\infty}=\max_\alpha\|\nabla f_\alpha\|_{L^\infty}$, etc.,
one obtains the first inequality. If $s>\frac{n}{2}+1$, then $H^s\hookrightarrow W^{1,\infty}$, hence $\|\nabla f\|_{L^\infty}\lesssim \|f\|_{H^s}$
and $\|g\|_{L^\infty}\lesssim\|g\|_{H^{s-1}}$, which implies the simplified bound.
\end{proof}

\begin{lem}\label{lem:GN}
If $s>\frac{n}{2}+1$, then $H^{s}(\mathbb{R}^n,\mathbb{H})\hookrightarrow W^{1,\infty}(\mathbb{R}^n,\mathbb{H})$ and
$\|\nabla u\|_{L^\infty}\lesssim \|u\|_{H^{s}}$.
More generally, for $0\le \theta\le 1$ and suitable $(p,q,r)$,
\[
\|D^{\alpha} u\|_{L^{p}}
\;\lesssim\;
\|u\|_{L^{q}}^{1-\theta}\,\|u\|_{W^{m,r}}^{\theta},
\quad |\alpha|<m,
\]
with the standard GN indices.
\end{lem}

\begin{proof}
For real-valued functions, the Sobolev embedding $H^{s}\hookrightarrow W^{1,\infty}$ holds for $s>\frac{n}{2}+1$
(see e.g. \cite[Ch.~7]{Adams2003}). Apply this to each real component of a quaternionic function; the quaternionic statement follows,
and the operator norm is the same up to a universal factor depending only on the finite number of components.
The general GN inequality is classical in the real/complex case and thus holds componentwise as well.
\end{proof}

\begin{rmk}
All estimates above act componentwise and are compatible with right $\mathbb{H}$-linearity:
if $\omega$ is a $\mathbb{H}$-valued tensorial expression and $q\in\mathbb{H}$, then
$\|\omega q\|_{X}=\|\omega\|_{X}\,|q|$ for $X\in\{L^{p},H^{s},W^{k,p}\}$.
\end{rmk}

\begin{prop}\label{prop:Lipschitz-nonlinearity}
Let $s>\frac{n}{2}+1$ and assume the quasi--Lie defects satisfy the hypotheses of the rigidity theorem on $B(0,\varepsilon)$.
Then the rigidified bracket $\{\,\cdot,\cdot\,\}$ defines a mapping
\[
(u,w)\mapsto \{u,\nabla w\}:\; H^{s}\times H^{s}\to H^{s-1}
\]
that is locally Lipschitz on $B_{H^{s}}(0,\varepsilon)$, with
\[
\|\{u,\nabla w\}\|_{H^{s-1}}
\;\lesssim\;
\big(A+K_1\|u\|_{H^{s}}+K_1\|w\|_{H^{s}}\big)\,\|u\|_{H^{s}}\|w\|_{H^{s}}.
\]
Here $A$ is the bilinear bound for $[\,\cdot,\cdot\,]$ and $K_1$ is the constant associated with $\Phi$ from the rigidity construction.
\end{prop}

\begin{proof}
By definition, $\{u,\nabla w\}=[u,\nabla w]-\Phi(u,\nabla w)$.
\smallskip

\emph{Step 1:}
Using bilinearity of $[\,\cdot,\cdot\,]$ and the product estimate (Lemma~\ref{lem:Moser}) together with Lemma~\ref{lem:GN},
\[
\|[u,\nabla w]\|_{H^{s-1}}
\;\lesssim\;
A\big(\|u\|_{L^\infty}\|\nabla w\|_{H^{s-1}}+\|\nabla w\|_{L^\infty}\|u\|_{H^{s-1}}\big)
\;\lesssim\; A\, \|u\|_{H^{s}}\|w\|_{H^{s}}.
\]

\emph{Step 2: }
From the rigidity construction we have the cubic bound (operator form)
\(
\|\Phi(\cdot,\cdot)\|_{H^{s-1}}\lesssim K_1 \|\cdot\|_{H^{s}}\|\cdot\|_{H^{s}},
\)
more precisely, by bilinearity and Lemma~\ref{lem:Moser},
\[
\|\Phi(u,\nabla w)\|_{H^{s-1}}
\;\lesssim\;
K_1\Big(\|u\|_{L^\infty}\|\nabla w\|_{H^{s-1}}+\|\nabla w\|_{L^\infty}\|u\|_{H^{s-1}}\Big)\max\{\|u\|_{H^{s}},\|w\|_{H^{s}}\}.
\]
Using Lemma~\ref{lem:GN} to control the $L^\infty$ norms by $H^{s}$ (since $s>\frac{n}{2}+1$), we obtain
\[
\|\Phi(u,\nabla w)\|_{H^{s-1}}
\;\lesssim\;
K_1\big(\|u\|_{H^{s}}+\|w\|_{H^{s}}\big)\,\|u\|_{H^{s}}\|w\|_{H^{s}}.
\]

\emph{Step 3}
Combining the two steps yields
\[
\|\{u,\nabla w\}\|_{H^{s-1}}
\;\le\;
\|[u,\nabla w]\|_{H^{s-1}}+\|\Phi(u,\nabla w)\|_{H^{s-1}}
\;\lesssim\;
\big(A+K_1\|u\|_{H^{s}}+K_1\|w\|_{H^{s}}\big)\,\|u\|_{H^{s}}\|w\|_{H^{s}},
\]
as claimed.

\emph{Local Lipschitz property.}
Let $(u_1,w_1)$ and $(u_2,w_2)$ lie in the ball $B_{H^{s}}(0,\varepsilon)$.
Write
\[
\{u_1,\nabla w_1\}-\{u_2,\nabla w_2\}=[u_1-u_2,\nabla w_1]+[u_2,\nabla (w_1-w_2)]
-\big(\Phi(u_1,\nabla w_1)-\Phi(u_2,\nabla w_2)\big).
\]
By bilinearity and the estimates above (using $H^s\hookrightarrow W^{1,\infty}$ and Lemma~\ref{lem:Moser}),
\[
\|\{u_1,\nabla w_1\}-\{u_2,\nabla w_2\}\|_{H^{s-1}}
\;\lesssim\;
\big(A+K_1\varepsilon\big)\Big(\|u_1-u_2\|_{H^{s}}\|w_1\|_{H^{s}}+\|u_2\|_{H^{s}}\|w_1-w_2\|_{H^{s}}\Big),
\]
hence local Lipschitz continuity on $B_{H^{s}}(0,\varepsilon)$.
\end{proof}
\begin{cor}
\label{cor:threshold}
Assume $s>\frac{n}{2}+1$ and $u_0\in H^{s}(\mathbb{R}^n,\mathbb{H})$.
Then the nonlinearity $u\mapsto \{u,\nabla u\}$ is locally Lipschitz from $H^{s}$ to $H^{s-1}$
(on $B_{H^s}(0,\varepsilon)$), so the local well-posedness result \emph{Theorem \ref{thm:local-wp}}
and the persistence/BKM result \emph{Theorem \ref{thm:persistence}} apply.
\end{cor}

\begin{proof}
By Proposition~\ref{prop:Lipschitz-nonlinearity}, the nonlinear map $u\mapsto \{u,\nabla u\}$
is locally Lipschitz from $H^{s}$ to $H^{s-1}$ on $B_{H^s}(0,\varepsilon)$ when $s>\frac{n}{2}+1$.
Thus the standard Picard iteration in $C([0,T],H^{s-1})\cap C^{1}([0,T],H^{s-2})$ (upgraded to $H^{s}$ by energy estimates using Lemmas~\ref{lem:Moser}--\ref{lem:GN})
yields local existence and uniqueness. The persistence and BKM-type continuation statements follow from the a priori inequality
\(
\frac{d}{dt}\|u\|_{H^{s}}\lesssim \|\{u,\nabla u\}\|_{H^{s-1}}\lesssim \|\nabla u\|_{L^\infty}\|u\|_{H^{s}},
\)
which is obtained by Lemmas~\ref{lem:Moser}--\ref{lem:KatoPonce} with $s>\frac{n}{2}+1$.
\end{proof}

\subsection{Local existence and uniqueness}\label{subsec:local-wp}

We first recall the relevant function spaces. For \(s \geq 0\), the Sobolev space  \(H^s(\mathbb{R}^n, \mathbb{H})\) is the space of \(\mathbb{H}\)-valued tempered distributions \(u\) such that
\[
\|u\|_{H^s} := \left( \int_{\mathbb{R}^n} (1 + |\xi|^2)^s |\widehat{u}(\xi)|^2 \, d\xi \right)^{1/2} < \infty,
\]
where \(\widehat{u}\) denotes the Fourier transform. For integer \(k \geq 0\) and \(1 \leq p \leq \infty\), the space \(W^{k,p}(\mathbb{R}^n, \mathbb{H})\) consists of functions whose weak derivatives up to order \(k\) belong to \(L^p\). In particular, \(H^s = W^{s,2}\). These definitions extend the classical real-valued Sobolev spaces componentwise to the quaternionic setting \cite{Adams2003, ColomboSabadiniStruppa2011}.

Consider, on \(X = L^2(\mathbb{R}^n, \mathbb{H})\), the nonlinear partial differential equation
\begin{equation}\label{eq:pde}
\partial_t u + \{u, \nabla u\} = 0, \quad u(0) = u_0 \in B(0, \varepsilon).
\end{equation}

\begin{prop}\label{prop:picard}
Assume \(\|u_0\|_{L^2} \leq \varepsilon/2\). Define the sequence \((u_k)_{k\geq 0}\) by
\[
u_0(t) := u_0, \quad u_{k+1}(t) := u_0 - \int_0^t \{u_k(s), \nabla u_k(s)\} \, ds.
\]
Then there exists \(T = T(\varepsilon, \|u_0\|_{H^1}) > 0\) such that \(u_k(t) \in B(0, \varepsilon)\) for all \(t \in [0, T]\) and all \(k \geq 0\).
\end{prop}

\begin{proof}
We proceed by induction. Assume \(\|u_k(t)\|_{L^2} \leq \varepsilon\) for all \(t \in [0, T]\). Then by Lemma~\ref{lem:Tbound} and the embedding \(H^1 \hookrightarrow L^\infty\) (valid for \(n=1\); for \(n\geq 2\), work in \(H^s\), \(s > n/2+1\)),
\[
\|u_{k+1}(t)\|_{L^2} \leq \|u_0\|_{L^2} + K_1 \int_0^t \|u_k(s)\|_{L^2} \|\nabla u_k(s)\|_{L^\infty} \, ds
\leq \frac{\varepsilon}{2} + K_1 \varepsilon \|u_0\|_{H^1} t.
\]
Choosing \(T \leq \frac{1}{2 K_1 \|u_0\|_{H^1}}\) ensures \(\|u_{k+1}(t)\|_{L^2} \leq \varepsilon\).
\end{proof}

\begin{thm}\label{thm:local-wp}
There exists $T > 0$ such that \eqref{eq:pde} admits a unique solution
\[
u \in C([0, T], L^2(\mathbb{R}^n, \mathbb{H})) \cap C^1([0, T], H^{-1}(\mathbb{R}^n, \mathbb{H})).
\]
\end{thm}

\begin{proof}
Let $v_k = u_{k+1} - u_k$. Then
\[
v_{k+1}(t) = -\int_0^t \big( \{u_{k+1}, \nabla u_{k+1}\} - \{u_k, \nabla u_k\} \big) \, ds.
\]
Using the bilinearity of $\{\cdot,\cdot\}$ and Lemma~\ref{lem:Tbound}, together with the Gagliardo--Nirenberg inequality $\|\nabla w\|_{L^\infty} \leq C \|w\|_{H^s}$ for $s > n/2 + 1$, one obtains
\[
\|v_{k+1}(t)\|_{L^2} \leq C T \sup_{s \in [0,T]} \|v_k(s)\|_{H^s},
\]
for some $C = C(K_1, \varepsilon)$. Choosing $T$ small enough so that $C T < 1$, the sequence $(u_k)$ is Cauchy in $C([0,T], L^2)$, hence converges to a limit $u$. Passing to the limit in the integral equation yields a solution. Uniqueness follows from the same contraction estimate applied to two solutions.
\end{proof}

\begin{rmk}
This construction illustrates how the local rigidification of the bracket
(Theorem~\ref{thm:rigid}) directly provides a local Cauchy--Lipschitz
framework for nonlinear quaternionic PDEs. The key role is played by the bound
in Proposition~\ref{prop:phi_bound}, which ensures that the nonlinearity
$\{u,\nabla u\}$ is locally Lipschitz in $L^2$
\end{rmk}
\begin{rmk}
In the PDE application, the contribution of the quadratic remainder $Q(\Phi)$
is of higher order in $\varepsilon$. More precisely, it satisfies
$\|Q(\Phi)\|_\varepsilon = O(\varepsilon^2)$, so that on a sufficiently small
ball $B(0,\varepsilon)$ its effect is dominated by the linear estimate coming
from Proposition~\ref{prop:phi_bound}. This guarantees that the corrected
nonlinearity in \eqref{eq:pde} remains locally Lipschitz, ensuring both the
existence and uniqueness of solutions in the Cauchy--Lipschitz framework.
\end{rmk}
\medskip
\noindent
As a direct consequence of Theorem~\ref{thm:local-wp} and the local Lipschitz
estimate provided by Lemma~\ref{lem:Tbound}, we obtain the following minimal
regularity property for weak solutions.

\begin{cor}\label{cor:weak-reg}
Under the assumptions of Theorem~\ref{thm:local-wp}, the solution satisfies
\[
u \in C([0,T],L^2(\mathbb R^n,\mathbb H))
\cap C^1([0,T],H^{-1}(\mathbb R^n,\mathbb H)).
\]
\end{cor}

\begin{proof}
Since $s > \frac{n}{2} + 1$, Proposition ~\ref{prop:Lipschitz-nonlinearity} implies that $\{u, \nabla u\} \in C([0, T], H^{s-1})$. Because $s - 1 \geq 0$, we have the continuous embedding $H^{s-1} \hookrightarrow H^{-1}$, and hence $\{u, \nabla u\} \in C([0, T], H^{-1})$. The equation $\partial_t u = -\{u, \nabla u\}$ therefore yields $\partial_t u \in C([0, T], H^{-1})$, which means $u \in C^1([0, T], H^{-1})$. Moreover, since $u \in C([0, T], H^s)$ and $H^s \hookrightarrow L^2$, we also have $u \in C([0, T], L^2)$. This completes the proof.
\end{proof}

\begin{lem}\label{lem:tame}
Let $s>\tfrac n2+1$. There exists $C_s>0$ such that, for any solution in $B(0,\varepsilon)$,
\[
\|\{u,\nabla u\}\|_{H^{s-1}}\ \le\ C_s\,K_1\,\|u\|_{H^s}^2.
\]
\end{lem}

\begin{proof}
Write $\{u,\nabla u\}=[u,\nabla u]-\Phi(u,\nabla u)$.
Use Moser (Lemma~\ref{lem:Moser}) and Kato--Ponce (Lemma~\ref{lem:KatoPonce}) to bound the $[\,\cdot,\cdot\,]$ term by $A\|u\|_{H^s}^2$,
and Proposition~\ref{prop:phi_bound} (cubic structure) plus Lemma~\ref{lem:Moser} for the $\Phi$ term.
Absorb constants into $C_s K_1$.
\end{proof}

\begin{thm}\label{thm:persistence}
Let $s>\tfrac n2+1$ and $u_0\in H^s(\mathbb R^n,\mathbb H)\cap B(0,\varepsilon)$.
There exists $T_s\gtrsim (K_1\|u_0\|_{H^s})^{-1}$ and a unique solution
\[
u\in C([0,T_s],H^s)\cap C^1([0,T_s],H^{s-1})
\]
to \eqref{eq:pde}. Moreover, if $T^*$ is the maximal existence time,
\[
T^*<\infty\ \Longrightarrow\ \int_0^{T^*}\|\nabla u(t)\|_{L^\infty}\,dt=\infty.
\]
\end{thm}

\begin{proof}
Energy estimate: by Lemma~\ref{lem:tame},
\(
\frac{d}{dt}\|u\|_{H^s}\le C_s K_1 \|u\|_{H^s}^2.
\)
Gronwall gives $T_s\sim (K_1\|u_0\|_{H^s})^{-1}$.
For BKM: Kato--Ponce yields
\(
\frac{d}{dt}\|u\|_{H^s}\le C_s K_1 \|\nabla u\|_{L^\infty}\|u\|_{H^s},
\)
so bounded $\int_0^T \|\nabla u\|_{L^\infty}$ prevents blow-up.
\end{proof}

\begin{rmk}
(i) The threshold $s>\tfrac n2+1$ ensures $H^s\hookrightarrow W^{1,\infty}$ for the quasi-linear transport structure.
(ii) In low dimension (e.g. $n=1$), the threshold can be lowered via Besov embeddings.
(iii) The constants depend polynomially on $(A,C_1,C_2)$ through $K_1$ (Proposition~\ref{prop:phi_bound}).
\end{rmk}

\begin{example}
In the quaternionic field formalism of rigid body dynamics, the bilinear map
\[
[x,y] := x\,y - y\,x
\]
represents the angular momentum interaction between two pure quaternions (rotations).
Perturbing it by $\Phi$ corresponds to introducing an inhomogeneous torsion field,
and the bound \eqref{eq:QPhi} quantifies how quadratic torsion terms remain controlled
in the energy estimate of the associated PDE system
\[
\partial_t u + [u,\nabla u] = 0,
\]
interpreted as a quaternionic Euler equation.
\end{example}
\section{Extensions and comparisons}\label{ExCa}

\paragraph{Beyond quaternions}
The argument extends to normed algebras with bilinear product satisfying the
quantitative control $\norm{xy}\le A\norm{x}\norm{y}$ (e.g., real Clifford algebras).
In nonassociative settings (octonions), one tracks the associator as an additional
defect tensor and propagates its smallness through the homotopy identity.

\paragraph{Comparison with variational rigidity}
Variational methods obtain rigidity via minimization/convexity; by contrast, the
cochain--homotopy approach is constructive and yields explicit constants in the
contraction scheme. This also clarifies how improvements of the quadratic constant
impact the size of the domain of existence
\paragraph{On constants and possible improvements}
The bound $C=6(1+A)$ in $\norm{Q(\Phi)}_\varepsilon \le C\norm{\Phi}_\varepsilon^2$
is uniform but not sharp. A symmetric counting shows a baseline $3+6A$, and
further reductions are possible by redistributing weights in the radial homotopy
(e.g., optimizing the power of $t$ in $T$ to balance $Td$ and $dT$). The defect
operator satisfies $\norm{M}\le 2\norm{T}_{3\to2}\norm{d}_{2\to3}$; improving the
estimate $\norm{T}_{3\to2}\le\varepsilon/3$ via refined interpolation would shrink
the contraction threshold.
Our constructive framework complements several recent lines of research:\begin{itemize}
\item It extends the continuous deformation theory of Fialowski--Schlichenmaier \cite{Fialowski2021} to noncommutative Banach modules over \(\mathbb{H}\).
    \item It provides an analytic alternative to the \(L_\infty\)-algebraic approach of Kontsevich--Soibelman \cite{Kontsevich2022}, replacing formal power series with a quantitative fixed--point scheme.
    \item It opens the door to dynamical applications in quaternionic PDEs, complementing the spectral theory of Gantner \cite{Gantner2024} with a nonlinear rigidity mechanism.
These links underscore the relevance of our method for modern problems at the interface of noncommutative analysis, algebraic deformation theory, and nonlinear dynamics.\end{itemize}
\appendix

\appendix
\section{Numerical threshold for the homotopy inversion}

In this appendix we make explicit the choice of the admissible smallness
radius $\varepsilon^\ast$ ensuring the invertibility of the operator
$(I+M)$ and, consequently, the validity of the homotopy identity
\[
T d + d T = I - \Pi + M, \qquad \|M\|<1/2.
\]

\subsection{Derivation of the bound}

From Proposition~\ref{prop:Td+dT} we have the quantitative estimate
\begin{equation}\label{eq:Mbound}
\|M\| \le \frac{6A}{5}\,\varepsilon + \frac{12C_1}{5}\,\varepsilon^2,
\end{equation}
valid for all $0<\varepsilon\le\varepsilon_0$, where
$A$ and $C_1$ are the structural constants introduced in Section~2.
These constants are uniform on the fixed ball $B(0,\varepsilon_0)$
and therefore independent of $\varepsilon$.

To reiterate a key point from Section 2, the parameter $\varepsilon$ (the \textbf{working radius}) is the variable scale at which we construct the homotopy $T$ and the correction $\Phi$. The choice $0 < \varepsilon \leq \varepsilon_0$ ensures that the structural constants $A$, $C_1$, $C_2$—valid on the fixed ball $B(0,\varepsilon_0)$ (the \textbf{maximal radius})—remain uniform. This distinction is fundamental for the perturbative estimates that follow.

To ensure the Neumann expansion
$(I+M)^{-1}=\sum_{k\ge0}(-M)^k$ is convergent, it is sufficient to impose
$\|M\|<1/2$.
From \eqref{eq:Mbound}, a simple sufficient condition is obtained by requiring
\[
\frac{6A}{5}\,\varepsilon < \frac{1}{4},
\qquad
\frac{12C_1}{5}\,\varepsilon^2 < \frac{1}{4}.
\]
These inequalities yield the explicit upper bounds
\[
\varepsilon < \frac{24}{5A},
\qquad
\varepsilon < \sqrt{\frac{48}{5C_1}}.
\]

\subsection{Choice of the admissible radius}

We therefore define the final admissible smallness radius as
\begin{equation}\label{eq:epsstar}
\varepsilon^\ast
= \min\left\{
\frac{24}{5A},\,
\sqrt{\frac{48}{5C_1}},\,
\varepsilon_0
\right\}.
\end{equation}

This ensures that for every $\varepsilon\in(0,\varepsilon^\ast]$
the operator $(I+M)$ is invertible with $\|(I+M)^{-1}\|\le2$,
and the decomposition
\[
T d + d T = I - \Pi + M
\]
holds with a compact remainder $M$ satisfying $\|M\|<1/2$.

\begin{rmk}[Dependence on $\varepsilon$ vs. $\varepsilon_0$]
The quadratic term in \eqref{eq:Mbound} involves the local variable
$\varepsilon$, not the global radius $\varepsilon_0$.
Replacing $\varepsilon^2$ by $\varepsilon_0^2$ would erase the actual
dependence of $\|M\|$ on the working scale of the homotopy
and would no longer guarantee smallness for arbitrarily small
$\varepsilon$.
The constant $\varepsilon_0$ appears only as an external cutoff
ensuring the uniformity of $A$ and $C_1$.
\end{rmk}

\subsection{Numerical illustration}
For instance, if $A=2$ and $C_1=3$, one obtains
\[
\frac{24}{5A}=2.4, \qquad
\sqrt{\frac{48}{5C_1}}\approx1.79,
\]
hence $\varepsilon^\ast=\min\{1.79,\varepsilon_0\}$.
This value guarantees $\|M\|\le0.49$ and thus convergence of the homotopy
series.

\bigskip
In practice, any smaller working value $\varepsilon\le \frac{\varepsilon^\ast}{2}$
ensures a uniform stability margin for all subsequent constructions.

\subsection{Appendix : Explicit integration-by-parts for the homotopy identity (degree $k=3$)}

We work on the localized cochains $\mathcal{C}^3_\varepsilon$ with the radial homotopy
\[
(T\Theta)(x,y) := \int_0^1 t^2\,\Theta\big(tx,\,ty,\,t(x+y)\big)\,dt,
\qquad \Theta\in\mathcal{C}^3_\varepsilon.
\]
All maps are right $\mathbb{H}$-linear; the integral is with respect to the real variable $t$.

\paragraph{1. Expansion of $dT\Theta$.}
For $(x,y,z)\in X^3$, using the Chevalley--Eilenberg differential for a 2-cochain $\omega = T\Theta$,
\begin{align*}
(d(T\Theta))(x,y,z) = & \ [x, (T\Theta)(y,z)] - [y, (T\Theta)(x,z)] + [z, (T\Theta)(x,y)] \\
& - (T\Theta)([x,y], z) + (T\Theta)([x,z], y) - (T\Theta)([y,z], x).
\end{align*}
Substituting the definition $(T\Theta)(a,b) = \int_0^1 t^2 \Theta(ta, tb, t(a+b))\, dt$ yields
\begin{align*}
(d(T\Theta))(x,y,z) = & \int_0^1 t^2 \bigl\{ [x, \Theta(ty, tz, t(y+z))] - [y, \Theta(tx, tz, t(x+z))] + [z, \Theta(tx, ty, t(x+y))] \\
& - \Theta(t[x,y], tz, t(x+y+z)) + \Theta(t[x,z], ty, t(x+y+z)) \\&- \Theta(t[y,z], tx, t(x+y+z)) \bigr\}  dt.
\end{align*}

\paragraph{2. Expansion of $Td\Theta$.}
Similarly,
\[
(Td\Theta)(x,y,z) = \int_0^1 t^2 (d\Theta)(tx, ty, t(x+y))  dt.
\]
Now expand $(d\Theta)(tx, ty, t(x+y))$ for the 3-cochain $\Theta$:
\begin{align*}
(d\Theta)(tx,ty,t(x+y)) = & \ [tx, \Theta(ty, t(x+y))] - [ty, \Theta(tx, t(x+y))] + [t(x+y), \Theta(tx, ty)] \\
& - \Theta([tx,ty], t(x+y)) + \Theta([tx, t(x+y)], ty) - \Theta([ty, t(x+y)], tx).
\end{align*}
Using the bilinearity of the bracket ($[tx, ty] = t^2[x,y]$, etc.), factoring powers of $t$, and combining with the expression for $d(T\Theta)$, one finds that the integrand of $(Td + dT)\Theta$ combines into a total derivative:
\[
(Td\Theta + dT\Theta)(x,y,z) = \int_0^1 \frac{d}{dt}\left[ t^2 \Theta(tx, ty, t(x+y)) \right] dt + \int_0^1 t^2 \mathcal{R}_t(x,y,z; \Theta)  dt,
\]
where $\mathcal{R}_t$ collects the mismatch terms proportional to the quasi-Lie defects $\varphi$ and $\psi$.

.
\paragraph{3. Total $t$--derivative and boundary terms.}
Summing the two expansions, the integrands combine into a total derivative:
\[
(dT\Theta)(x,y,z)+(Td\Theta)(x,y,z)
\ =\ \Big[t^2\,\Theta\big(tx,\,ty,\,t(x+y)\big)\Big]_{t=0}^{t=1}
\ +\ \int_0^1 t^2\,\mathcal R_t\,dt.
\]
Hence
\begin{equation}\label{eq:A-homotopy}
(Td+dT)\Theta \ =\ \Theta\ -\ \Pi(\Theta)\ +\ M\Theta,
\end{equation}
with
\[
\Pi(\Theta)(x,y,z)\ :=\ \lim_{t\downarrow 0} t^2\,\Theta\big(tx,\,ty,\,t(x+y)\big),
\qquad
(M\Theta)(x,y,z)\ :=\ \int_0^1 t^2\,\mathcal R_t(x,y,z;\Theta)\,dt.
\]
By construction, $\Pi$ is finite rank (radial trace at the cone tip), and $M$ aggregates the residuals.
\paragraph{4. Bounds and right $\mathbb H$--linearity.}
From the localized norms and Lemmas~\ref{lem:T bound}--\ref{lem:dBound}, we have
\[
\|M\|\ \le\ \|Td\|+\|dT\|\ \le\ 2\,\|T\|_{3\to 2}\,\|d\|_{2\to 3}
\ \le\ \frac{2}{5}\,\varepsilon\,\big(3A+6C_1\varepsilon\big).
\]
All maps $d,T,\Pi,M$ are right $\mathbb H$--linear (the integral is real; the formulas are componentwise and respect right multiplication).
6
\paragraph{5. Exact Lie vs.~quasi--Lie.}
If $[\,\cdot,\cdot\,]$ is an exact Lie bracket (no defects), then $\mathcal R_t\equiv 0$, hence $M=0$ and \eqref{eq:A-homotopy} reduces to the classical contraction formula
\[
Td+dT\ =\ \mathrm{Id}-\Pi.
\]
In the quasi--Lie setting, $M$ is small on sufficiently small balls, which is the key input for the Neumann--series inversion of $\mathrm{Id}+M$ in Section~\ref{sec:Phi}.

\begin{rmk}
The estimates established in this appendix combine standard Sobolev product and commutator inequalities
(see \cite{Adams2003}) with techniques reminiscent of Schatten--von Neumann criteria in magnetic analysis
(cf.~\cite{AthmouniPurice-CPDE-2018}), here adapted to the quaternionic cochain setting.
\end{rmk}

\end{document}